\documentclass[12pt,leqno]{article}

\usepackage{fullpage}
\usepackage{amssymb}
\usepackage{amsmath}
\usepackage{amsthm}
\newtheorem{theorem}{Theorem}
\newtheorem{problem}{Problem}
\newtheorem{conjecture}{Conjecture}
\newtheorem{definition}{Definition}

\newtheorem{lemma}{Lemma}
\newtheorem{remark}{Remark}
\begin{document}
\title{\large\bf On Sequences  Containing \\ at Most 4 Pairwise Coprime Integers\footnote{This work was supported by the National Natural
Science Foundation of China, Grant No 11071121. }}
\date{}
\author{ Yong-Gao Chen\footnote{Email: ygchen@njnu.edu.cn}  and Xiao-Feng Zhou
\\
\small School of Mathematical Sciences, Nanjing Normal
University,\\ \small Nanjing 210046, P. R. CHINA}
 \maketitle

\begin{abstract} Let $f(n,k)$ be the largest number of positive integers not exceeding $n$
from which one cannot select $k+1$ pairwise coprime integers, and
let $E(n,k)$ be the set of positive integers which do not exceed
$n$ and can be divided by at least one of $p_1, p_2, \ldots ,
p_k$, where $p_i$ is the $i$-th prime. In 1962, P. Erd\H os
conjectured that $f(n,k)=|E(n,k)|$ for all  $n\ge p_k$. In 1973,
S.  L. G. Choi proved that the  conjecture is true for $k=3$. In
1994, Ahlswede and Kachatrian disproved the conjecture for
$k=212$. In this paper we prove that, for $n\ge 49$, if $A(n,4)$
is a set of positive integers not exceeding $n$ from which one
cannot select $5$ pairwise coprime integers and  $|A(n,4)|\ge
|E(n,4)|$, then $A(n,4)=E(n,4)$. In particular, the conjecture is
true for $k=4$. Several open problems and conjectures are posed
for further research.

\end{abstract}



\section{ Introduction}

Let $p_i$ be the $i$-th prime. Let $f(n,k)$ be the largest number
of positive integers not exceeding $n$ from which one cannot
select $k+1$ pairwise coprime integers, and let $E(n,k)$ be the
set of positive integers not exceeding $n$ and divisible by at
least one of $p_1, p_2, \ldots , p_k$. It is clear that $f(n,k)\ge
|E(n,k)|$ for all $n,k$ and $f(n,k)=n=|E(n,k)|+1$ for $n<p_k$. In
1962, P. Erd\H os \cite{Erdos1962}\cite{Erdos1965} conjectured
that $f(n,k)=|E(n,k)|$ for all  $n\ge p_k$. It is easy to see that
the conjecture is true for $k=1,2$. Let $A(n,k)$ be a set of
positive integers not exceeding $n$ from which one cannot select
$k+1$ pairwise coprime integers. In 1973, S.  L. G. Choi
\cite{Choi} proved that for $k=3$
 the  conjecture is true  and  for $n\ge 150$, if $|A(n,3)|=|E(n,3)|$,
then $A(n,3)= E(n,3)$ (it is remarked that it is possible to prove
this for $n\ge 92$). In 1985, the conjecture for $k=3$ is also
proved by Szab\'o and T\'oth. In 1994, R. Ahlswede and L. H.
Khachatrian \cite{Ahlswede} proved that
 the  conjecture is false for $k=212$. In the sequel, Erd\H os relaxed his conjecture to:
 For each $k$ there are only finitely many $n$ satisfying $f(n,k)\not= |E(n,k)|$.

 In this paper we prove the following results. In particular, the original conjecture of Erd\H os is true for $k=4$.

\begin{theorem}\label{mainthm1} Let $n\ge 55$. If  $|A(n,3)|\ge
|E(n,3)|$, then $A(n,3)= E(n,3)$. Furthermore 55 is the best
possible.

\end{theorem}

\begin{theorem}\label{mainthm2}  Let $n\ge 7$.  If $|A(n,4)|\ge
|E(n,4)|$, then $|A(n,4)|= |E(n,4)|$ for $7\le n\le 48$ and
$A(n,4)= E(n,4)$ for $n\ge 49$. Furthermore 49 is the best
possible.

In particular, $f(n,4)=|E(n,4)|$ for all $n\ge 7$.
\end{theorem}

Let $F_k$ be the set of integers which can be divided by at least
one of $p_1, p_2,\dots , p_k$. Let $[m,n]=\{ m, m+1,\ldots , n\}$.
By definition, $F_k\cap [1, n]=E(n, k)$ for all $n\ge 1$.

 We pose the following  conjectures:

\begin{conjecture}\label{conj1} Let $k\ge 3$. For $n=p_1p_2\dots p_k-p_{k+1}$ we have
$f(n,k)=|E(n,k)|$.\end{conjecture}

\begin{conjecture} \label{conj2}Let $k\ge 3,l$ be two positive integers. For any
integer $a$ with $-p_{k+1}+1\le a\le p_1p_2\dots p_k-p_{k+1}$, if
$B_{k,l}(a)$ is a subset of $[-p_{k+1}+1, a]$ such that one cannot
select $k+1$ pairwise coprime integers from $\{ p_1p_2\dots p_kl+b :
b\in B_{k,l}(a)\}$,
 then
$$ |B_{k,l}(a)|\le |F_k\cap [-p_{k+1}+1, a]|.$$
\end{conjecture}

\begin{conjecture} For any integer $k\ge 3$ there exists an
integer $n_k$ such that for $n\ge n_k$, if  $|A(n,k)|\ge |E(n,k)|$,
then $A(n,k)=E(n,k)$.
\end{conjecture}

\begin{remark} Let $E(k)$ denote the least such $n_k$. By Theorems
\ref{mainthm1} and \ref{mainthm2} we have $E(3)=55$ and $E(4)=49$.
Similarly, one may derive that $E(1)=4$ and $E(2)=9$.

Let $A(p_k^2-1, k)=(F_k\cap [1, p_k^2-1]\cup \{ p_{k+1}\}
)\setminus \{ p_k\} $. Then $A(p_k^2-1, k)\not= E(p_k^2-1, k)$ and
$|A(p_k^2-1, k)|= |E(p_k^2-1, k)|$. So $E(k)\ge p_k^2$.
\end{remark}

\begin{theorem}\label{conj3-4thm} Conjectures \ref{conj1} and
\ref{conj2} are true for $k=3,4$.\end{theorem}

\begin{remark} Conjecture \ref{conj1} for $k=3$ follows from the
original conjecture of Erd\H os for $k=3$. For given $k>4$,
Conjecture \ref{conj1} can be verified by the method in Section
\ref{k=4}.\end{remark}

\begin{theorem}\label{generalthm} Suppose that $n_0,l_0$ are two positive
integers such that Conjecture \ref{conj2} is true for all $l\ge
l_0$, and
 $|A(n,k)|\ge |E(n,k)|$ implies that
$A(n,k)=E(n,k)$ for $n_0\le n\le p_1p_2\dots p_kl_0-p_{k+1}$. Then
$|A(n,k)|\ge |E(n,k)|$ implies that $A(n,k)=E(n,k)$ for all $n\ge
n_0$.\end{theorem}

In the last section we pose several open problems and a conjecture
for further research.

\section{Proof of Theorem \ref{generalthm}}\label{secgeneralthm}

We use induction on $n$. By the condition we have that
 $|A(n,k)|\ge |E(n,k)|$ implies that
$A(n,k)=E(n,k)$ for $n_0\le n\le p_1p_2\dots p_kl_0-p_{k+1}$.
Suppose that  $|A(n,k)|\ge |E(n,k)|$ implies that $A(n,k)=E(n,k)$
for all $n_0\le n<m$ $(m>p_1p_2\dots p_kl_0-p_{k+1})$ and
Conjecture \ref{conj2} is true for all $l\ge l_0$. Then there
exist two integers $l, a$ with $-p_{k+1}+1\le a\le p_1p_2\dots
p_k-p_{k+1}$ such that
$$m=p_1p_2\dots p_kl+a,\quad  l\ge l_0.$$
Let
$$B_{k,l}(a)=\{ b:  p_1p_2\dots p_kl+b \in A(m,k),-p_{k+1}+1\le b\le a \}.$$
Then  one cannot select $k+1$ pairwise coprime integers from the set
$\{ p_1p_2\dots p_kl+b : b\in B_{k,l}(a)\}$. By Conjecture
\ref{conj2} we have \begin{eqnarray*}&&|A(m,k)\cap [p_1p_2\dots
p_kl-p_{k+1}+1, m]|\\
&=& |B_{k,l}(a)|\le |F_k\cap [-p_{k+1}+1, a]| \\
&=& |F_k \cap [p_1p_2\dots p_kl-p_{k+1}+1,m] |.\end{eqnarray*} If
$|A(m,k)|\ge |E(m,k)|=|F_k\cap [1, m]|$, then
$$ |A(m,k)\cap [1, p_1p_2\dots
p_kl-p_{k+1}]|\ge |F_k\cap [1, p_1p_2\dots p_kl-p_{k+1}]|.$$ By the
induction hypothesis we have
$$A(m,k)\cap [1, p_1p_2\dots
p_kl-p_{k+1}]=E(p_1p_2\dots p_kl-p_{k+1},k).$$ Thus, $p_1,p_2,\dots
, p_k\in A(m,k)$.  If $(i, p_1\cdots p_k)=1$, then $i\notin A(m,k)$.
Hence $A(m,k)\subseteq E(m,k)$. Since $|A(m,k)|\ge |E(m,k)|$, we
have $A(m,k)= E(m,k)$.
 This completes the proof of
Theorem \ref{generalthm}.

\section{Preliminary Lemmas for Conjecture \ref{conj2}}\label{seca}

Let $k\ge 3,l$ be two positive integers and let
$$T_{k}=\{ a : -p_{k+1}+1\le a\le p_1p_2\dots p_k-p_{k+1}, (a, p_1p_2\dots
p_k)=1\} .$$
%Let $T=\{ a_1<a_2<\cdots <a_t\} $. Then $a_1=-1$,$a_2=1$, $a_3=p_{k+1}$, $\dots$.

For any integer $a$, let $B_{k,l}(a)$ be a subset of $[-p_{k+1}+1,
a]$ such that one cannot select $k+1$ pairwise coprime integers from
$\{ p_1p_2\dots p_kl+b : b\in B_{k,l}(a)\}$.

\begin{lemma}\label{lemma1}
Conjecture \ref{conj2} is true for $a\in \{ -1, 1\} $.
\end{lemma}

\begin{proof} Let $a\in \{ -1, 1\} $. If $b\in [-p_{k+1}+1, a]$ and $b\notin \{ -1, 1\}
$, then $b\in F_k\cap [-p_{k+1}+1, a]$. Since the integers in the
set
$$\{ p_1p_2\dots p_kl+b : b=-p_k, -p_{k-1}, \dots , -p_1, -1, 1\}$$
are coprime each other, we have
$$|B_{k,l}(a)\cap \{ -p_k, -p_{k-1}, \dots , -p_1, -1, 1\} |\le k=|\{ -p_k, -p_{k-1}, \dots ,
-p_1\} |.$$ Therefore $$ |B_{k,l}(a)|\le |F_k\cap [-p_{k+1}+1,
a]|.$$ This completes the proof of Lemma \ref{lemma1}.
 \end{proof}

\begin{lemma}\label{lemma3} Suppose that $-p_{k+1}+1 < b \le a\le p_1p_2\dots p_k-p_{k+1}$,
 Conjecture \ref{conj2} is true for $b-1$ and $|B_{k,l}(a)\cap [b, a]|\le |F_k\cap [b, a]|$. Then $$|B_{k,l}(a)|\le  |F_k\cap [-p_{k+1}+1,
a]|.$$ In particular, if Conjecture \ref{conj2} is true for $a-1$
and $a\notin B_{k,l}(a)\cup T_k$, then
$$|B_{k,l}(a)|\le  |F_k\cap [-p_{k+1}+1,
a]|.$$
\end{lemma}

\begin{proof} By the assumption we have
$$|B_{k,l}(a)\cap [-p_{k+1}+1, b-1]|\le  |F_k\cap [-p_{k+1}+1,
b-1]|.$$ Since $|B_{k,l}(a)\cap [b, a]|\le |F_k\cap [b, a]|$, we
have
$$|B_{k,l}(a)|\le  |F_k\cap [-p_{k+1}+1,
a]|.$$ If $a\notin B_{k,l}(a)\cup T_k$, then either $a\notin
B_{k,l}(a)$ or $a\in F_k$. Thus $|B_{k,l}(a)\cap [a, a]|\le
|F_k\cap [a, a]|$. Hence
$$|B_{k,l}(a)|\le  |F_k\cap [-p_{k+1}+1,
a]|.$$ This completes the proof of Lemma \ref{lemma3}.
\end{proof}

\begin{definition} Let $l$ be an integer and $a_1<a_2<\cdots
< a_t$ be integers.  $\{ a_1, a_2, \ldots , a_t\} $ is called a
$l$-good set if $(p_1p_2\dots p_kl+a_i, a_i-a_j)=1(1\le i<j\le
t)$.
\end{definition}

$\{ a_1, a_2, \ldots , a_t\} $ is  called \emph{ a good set } if
$(a_i, a_j, p_1p_2\dots p_k)=1$ for all $1\le i<j\le t$ and
$a_i-a_j$ has no prime factors more than $p_k$ for any $i\not= j$.
It is clear that a good set $\{ a_1, a_2, \ldots , a_t\} $ is also
a $l$-good set for any integer $l$.

\begin{lemma}\label{goodset} If $\{ a_1, a_2, \ldots , a_t\} $ is  a $l$-good
set,  then $p_1p_2\dots p_kl+a_1, p_1p_2\dots p_kl+a_2, \ldots ,
p_1p_2\dots p_kl+a_t$ are pairwise coprime.

In particular, if $\{ a_1, a_2, \ldots , a_t\} $ is a $l$-good
set, then $$|B_{k,l}(a)\cap \{ a_1, a_2, \ldots , a_t\} |\le k.$$
\end{lemma}

 The proof follows from the definition of l-good set immediately.

\begin{lemma}\label{lemma2} Conjecture \ref{conj2} is true for $a\in \{ p_{k+1}, p_{k+2}\} $.
\end{lemma}

\begin{proof}For a
positive integer $n$ we use $P(n)$ to denote the largest prime
factor of $n$. By the definition, $\{ p_1, p_2, \dots , p_{k+1}\}
$ and $\{ p_1^2, p_2, \dots , p_{k+1}\} $  are both good sets. By
the Bertrand's postulate we have $p_{k+2}<2p_{k+1}$. For $2\le
i\le k+1$ we have $P(p_{k+2}-p_i)<\frac 12 p_{k+2}<p_{k+1}$.

{\bf Case 1:} $p_{k+2}-2$ is composite. We have $P(p_{k+2}-p_1)<
\frac 13 p_{k+2}<p_{k+1}$. So $\{ p_1, p_2, \dots ,
p_{k+1},p_{k+2}\} $ is a good set. By Lemma \ref{goodset} we have
$$|B_{k,l}(a)\cap \{ p_1, p_2, \dots ,
p_{k+1},p_{k+2}\} |\le k=|\{ p_1, p_2, \dots ,p_k\}|.$$ If $p_1\le
b<a$ with $b\notin \{ p_1, p_2, \dots , p_{k+1},p_{k+2}\}$, then
$$b\in F_k\setminus \{ p_1, p_2, \dots ,p_k\}.$$ Hence
$$ |B_{k,l}(a)\cap [p_1, a]|\le |F_k\cap [p_1, a]|.$$

{\bf Case 2:} $p_{k+2}-2$ is a prime. Then $p_{k+2}-2^2$ is
composite. $P(p_{k+2}-p_1^2)< \frac 13 p_{k+2}<p_{k+1}$. So $\{
p_1^2, p_2, \dots , p_{k+1},p_{k+2}\} $ is a good set. Similar to
Case 1 we have
$$ |B_{k,l}(a)\cap [p_2, a]|\le |F_k\cap [p_2, a]|.$$

By Lemma \ref{lemma3}, Conjecture \ref{conj2} is true for $a\in \{
p_{k+1}, p_{k+2}\} $. This completes the proof of Lemma
\ref{lemma2}.
 \end{proof}

%$b$ be an integer with

\begin{lemma}\label{4-set} Suppose that $a\in T_k\cap B_{k,l}(a)$. Let  $-p_{k+1}+1\le b<a$, $U_1,\ldots , U_r$
be pairwise disjoint subsets of $F_k\cap [b, a]$ with
$|U_i|=k-1,k$ $(1\le i\le r)$, and $S_1, S_2, \dots , S_r$ be
subsets of $T_k$ with
\begin{equation}\label{U-S0}B_{k,l}(a)\cap T_k\cap  [b,a]\subseteq  S_1\cup S_2\cup \cdots \cup S_r.\end{equation} Suppose that

(i) if $|U_i|=k$, then  $ U_i\cup S_i$ is a $l$-good set;

(ii) if $|U_i|=k-1$, then $a\notin S_i$ and  $ U_i\cup S_i \cup \{
a \}$ is a $l$-good set.

 Then $|B_{k,l}(a)\cap [b, a]|\le |F_k\cap [b,
a]|$.\end{lemma}

\begin{proof}
 By Lemma \ref{goodset}
we have
$$|B_{k,l}(a)\cap (U_i\cup S_i)|\le k=|U_i|,\quad \text{ if }
|U_i|=k,$$
$$|B_{k,l}(a)\cap (U_i\cup S_i)|=|B_{k,l}(a)\cap (U_i\cup S_i\cup \{ a\})|-1\le k-1=|U_i|,\quad \text{ if }
|U_i|=k-1.$$ Hence
\begin{equation}\label{U-S1} \sum_{i=1}^r |B_{k,l}(a)\cap (U_i\cup S_i)|\le
\sum_{i=1}^r|U_i|=|U_1\cup U_2\cup\cdots\cup U_r|.\end{equation} By
(\ref{U-S0}) we have
\begin{equation}\label{U-S2}B_{k,l}(a)\cap T_k\cap  [b,a]=B_{k,l}(a)\cap [b, a]\cap (S_1\cup  \cdots
\cup S_r).\end{equation} By (\ref{U-S2}) we have
\begin{eqnarray*}&&|B_{k,l}(a)\cap [b, a]|\\
&=&|B_{k,l}(a)\cap T_k\cap  [b,a]|+|B_{k,l}(a) \cap [b, a]\cap (U_1\cup \cdots \cup
U_r)|\\
&& +|B_{k,l}(a)\cap [b, a]\cap (F_k\setminus (U_1\cup \cdots \cup U_r))|\\
&= & |B_{k,l}(a)\cap [b, a]\cap (S_1\cup  \cdots \cup
S_r)|+|B_{k,l}(a) \cap [b,
a]\cap (U_1\cup \cdots \cup U_r)|\\
&&+|B_{k,l}(a)\cap [b, a]\cap (F_k\setminus (U_1\cup \cdots \cup U_r))|\\
&=& |B_{k,l}(a)\cap [b, a] \cap (U_1\cup S_1\cup\cdots\cup U_r\cup S_r)|\\
&&+|B_{k,l}(a)\cap [b, a]\cap (F_k\setminus (U_1\cup \cdots \cup U_r))|\\
&\le & |U_1\cup U_2\cup\cdots\cup U_r|+|[b, a]\cap (F_k\setminus
(U_1\cup \cdots \cup U_r))|\\
&=&|F_k\cap [b, a]|.\end{eqnarray*} This completes the proof of
Lemma \ref{4-set}.
\end{proof}

{\bf Remark.} In the application, we need only to give $U_i\cup
S_i$ $(1\le i\le r)$. It is natural to take $b$ being the least
integer in the set $U_1\cup S_1\cup\cdots\cup U_r\cup S_r$.

\section{Conjecture \ref{conj2} for $k=3$}

In this section we prove that Conjecture \ref{conj2} for $k=3$ is
true. We have $T_3=\{-1,1,7,11,13, 17, 19, 23\}$. We use induction
on $a\in [-6, 23]$.

It is clear that $|B_{3,l}(-6)|\le 1=|F_3\cap [-6, -6]|$. Suppose
that  $a\in [-5, 23]$ and  $|B_{3,l}(a')|\le |F_3\cap [-6, a']|$
for all $a'\in [-6, a-1]$ and all $B_{3,l}(a')$. By Lemma
\ref{lemma3} we may assume that $a\in T_3\cap B_{3,l}(a)$.
 By Lemmas \ref{lemma1} and \ref{lemma3}, Conjecture \ref{conj2}  is true for
$k=3$ and $a\in \{-1,1,7,11\}$.

By Lemma \ref{lemma2} and the remark of Lemma \ref{4-set}, it is
enough to give $U_i\cup S_i$ $(1\le i\le r)$ which satisfy Lemma
\ref{4-set} for $a\in \{ 13, 17, 19, 23\}$.

{\bf Case 1:}  $a= 13, 17,19$. Let
$$U_1\cup S_1= \{8, 9,5, 7, 11, 13, 17\},$$
$$U_2\cup S_2=\{ 4, 3, -5, 1, -1,19\}. $$

{\bf Case 2:}  $a=23$. Let
$$U_1\cup S_1= \{8, 3,5,  11, 23\},$$
$$U_2\cup S_2=\{ -2, -3, -5, -1, 1,7\}, $$
$$U_3\cup S_3=\{ 21, 22, 13, 17, 19\}.$$

This completes the proof of Conjecture \ref{conj2} for $k=3$.

\section{Proof of Theorem \ref{mainthm1}}

Let
$$A(54, 3)=(F_3\setminus \{ 5, 25\} ) \cup \{ 7, 49\} .$$
Then $A(54, 3)$ does not contain $4$ pairwise coprime integers,
$|A(54, 3)|=|F_3\cap [1, 54]|$ and $A(54, 3)\not=F_3\cap [1, 54]$.

We will prove that for $55\le n\le 83=30\times 3-7$, if $|A(n,
3)|\ge |F_3\cap [1, n]|$, then $A(n, 3)=F_3\cap [1, n]$.

 Let $H_1$ be the set of all primes together with $1$. Let
$$H_2=\{ 2^2, 3^2, 5^2, 77\} , $$
$$H_3=\{ 2^3, 3^3, 55,7^2\} ,$$
$$H_4=F_3\setminus \{ 2, 3, 5, 2^2, 3^2, 5^2, 2^3, 3^3, 55\} .$$
For $55\le n\le 83$, we have
$$|A(n,3)\cap H_1|\le 3,$$
$$|A(n,3)\cap H_2|\le 3,$$
$$|A(n,3)\cap H_3|\le 3,$$
$$|A(n,3)\cap H_4|\le |F_3\cap [1,n]|-9,$$
$$[1,n]\subseteq H_1\cup H_2\cup H_3\cup H_4.$$
Hence
$$|A(n,3)|\le \sum_{i=1}^4 |A(n,3)\cap H_i|\le  |F_3\cap [1,n]|.$$
Since $|A(n, 3)|\ge |F_3\cap [1, n]|$, we have $|A(n, 3)|=
|F_3\cap [1, n]|$, $|A(n,3)\cap H_i|=3 (1\le i\le 3)$ and
$|A(n,3)\cap H_4|= |F_3\cap [1,n]|-9$. By $|A(n,3)\cap H_4|=
|F_3\cap [1,n]|-9$ we have $2^4, 51\in A(n,3)$. Since
$2^4,51,7,11$ are coprime each other,we have either $7\notin
A(n,3)$ or $11\notin A(n,3)$. By $|A(n,3)\cap H_2|=3$ we have
$A(n, 3)\cap H_1\subseteq \{ 2, 3, 5, 7, 11\} $.

If $7\in A(n,3)$ or $11\in A(n,3)$, then by $2^4,3\times 17\in
A(n,3)$ we have $5^2, 5\times 13\notin A(n,3)$. By $|A(n,3)\cap
H_2|= 3$ we have $A(n,3)\cap H_2=\{ 2^2, 3^2, 77\}$. So $n\ge 77$.
By $|A(n,3)\cap H_4|= |F_3\cap [1,n]|-9$ we have $5\times 13\in
A(n,3)$, a contradiction. Hence $7, 11\notin A(n,3)$. By
$|A(n,3)\cap H_1|= 3$ we have $A(n,3)\cap H_1=\{ 2, 3, 5\}$. Thus,
if $(m, 30)=1$, then $m\notin A(n,3)$. This implies that
$A(n,3)\subseteq F_3$. By $|A(n, 3)|\ge |F_3\cap [1, n]|$, we have
$A(n, 3)=F_3\cap [1, n]$.

Now Theorem \ref{mainthm1} follows from Theorem \ref{generalthm} and
Conjecture \ref{conj2}.

\section{Conjecture \ref{conj2} for $k=4$}

In this section we prove that Conjecture \ref{conj2} for $k=4$ is
true. We have
\begin{eqnarray*}T_4&=&\{ m : -10\le m\le 199, (m, 210)=1\} \\
&=&\{-1, 1, 11, 13, 17, 19, 23, 29, 31, 37, 41, 43, 47, 53, 59, 61,
67, \
71, 73, 79,\\
&&83, 89, 97, 101, 103, 107, 109, 113, 121, 127, 131, 137, \ 139,
143, 149, 151, \\ &&157, 163, 167, 169, 173, 179, 181, 187, 191,
193, \ 197, 199\}.\end{eqnarray*} We use induction on $a\in [-10,
199]$.

It is clear that $|B_{4,l}(-10)|\le 1=|F_4\cap [-10, -10]|$.
Suppose that  $a\in [-9, 199]$ and  $|B_{4,l}(a')|\le |F_4\cap
[-10, a']|$ for all $a'\in [-10, a-1]$ and all $B_{4,l}(a')$. By
Lemma \ref{lemma3} we may assume that  $a\in T_4\cap B_{4,l}(a)$.
 By Lemmas \ref{lemma1} and \ref{lemma3}, Conjecture \ref{conj2}  is true for
$k=4$ and $a\in \{-1,1,11,13\}$.

By Lemma \ref{lemma2} and the remark of Lemma \ref{4-set}, it is
enough to give $U_i\cup S_i$ $(1\le i\le r)$ which satisfy Lemma
\ref{4-set} for $a\in T_4$ with $a\ge 17$.

\noindent {\bf Case 1:} $a= 17, 19, 23, 29, 31$. Let
$$U_1\cup S_1=\{ -7, 5, 9, 8 \} \cup \{  -1,11,13,17,23,29 \},$$
$$U_2\cup S_2=\{7, -5, 3, 4 \} \cup \{  1, 19,31\} .$$

\noindent {\bf Case 2:} $a=37, 41$. Let
$$U_1\cup S_1=\{ -7,5,-9,-4 \} \cup \{-1,1,11,41   \},$$
$$U_2\cup S_2=\{7,-5,9,16 \} \cup \{13,19,23,37  \} ,$$
$$U_3\cup S_3=\{25,27,32 \} \cup \{ 17,29,31 \} .$$
\noindent {\bf Case 3:}  $a=43, 47  $. Let
$$U_1\cup S_1=\{-7,5,-3,2  \} \cup \{-1,11,17,47  \},$$
$$U_2\cup S_2=\{7,-5,3,-2 \} \cup \{ 1,13,19,43 \} ,$$
$$U_3\cup S_3=\{35,33,38 \} \cup \{23,29,31,37,41  \} .$$

\noindent {\bf Case 4:} $a=53$. Let
$$U_1\cup S_1=\{ 49, 25, 39, 46\} \cup \{ 37,41,43,53 \} ,$$
$$U_2\cup S_2=\{ 35, 33, 38\} \cup \{ 17,23,29,47 \} ,$$
$$U_3\cup S_3=\{ 7,-5,3,4\} \cup \{ -1,1,11,13,19,31 \} .$$

\noindent {\bf Case 5:} $a=59, 61, 67$. Let
$$U_1\cup S_1=\{ 49, 55, 51, 58 \} \cup \{  53,59,61,67\} .$$

\noindent {\bf Case 6:} $a=71$.

If $11\nmid 210l+71$, then
$$U_1\cup S_1=\{ 49, 55, 51,46 \} \cup \{ 47,53,61,67,71 \} ,$$
$$U_2\cup S_2=\{ 65,57,64\} \cup \{ 59 \} .$$

If $11\mid 210l+71$, then
$$U_1\cup S_1=\{ 7, 55, 57,22\} \cup \{ 43,47,67,71 \} ,$$
$$U_2\cup S_2=\{ 65,69,68\} \cup \{ 41,53,59,61\} ,$$
$$U_3\cup S_3=\{ 49, 25, 39, 34\} \cup \{ 19,29,31,37 \} ,$$
$$U_4\cup S_4=\{ -7,5,8,3\} \cup \{-1,1,11,13,17,23 \} .$$

\noindent {\bf Case 7:} $a=73$. Let
$$U_1\cup S_1= \{49,55,57,64  \} \cup \{59,61,67,73 \} ,$$
$$U_2\cup S_2= \{65,69,68  \} \cup \{53,71 \} .$$

\noindent {\bf Case 8:} $a=79, 83, 89$. Let
$$U_1\cup S_1=\{ 49, 55, 51, 58 \} \cup \{  53,59,61,67\} ,$$
$$U_2\cup S_2=\{ 77, 65, 69, 74 \} \cup \{  71,73,79,83,89\} .$$

\noindent {\bf Case 9:} $a= 97$.
 Let
$$U_1\cup S_1=\{ 91,95,93,94\} \cup \{ 97\} .$$

\noindent {\bf Case 10:} $a=101, 103, 107, 109, 113$. Let

$$U_1\cup S_1=\{ 77,95,93,92 \} \cup \{ 83,101,107,113\} ,$$
$$U_2\cup S_2=\{ 91,85,99,94 \} \cup \{ 79,89,97,103,109\} .$$

\noindent {\bf Case 11:} $a=121,127$. Let
$$U_1\cup S_1=\{ 119, 115,117,118 \} \cup \{  121,127\} .$$

\noindent {\bf Case 12:} $a=131$. Let
$$U_1\cup S_1=\{ 119,115,117,122 \} \cup \{ 121,127,131\} .$$

\noindent {\bf Case 13:} $a=137,139,143$. Let
$$U_1\cup S_1=\{ 133,125,129,134 \} \cup \{ 127,131,137,139,143\} .$$

\noindent {\bf Case 14:} $a=149,151,157$. Let
$$U_1\cup S_1=\{ 119,125,129,128 \} \cup \{ 131,137,143,149\} ,$$
$$U_2\cup S_2=\{ 133,145,141,136 \} \cup \{ 121,127,139,151,157,181\} .$$

\noindent {\bf Case 15:} $a=163,167$. Let
$$U_1\cup S_1=\{ 161,155,153,158 \} \cup \{ 157,163,167\} .$$

\noindent {\bf Case 16:} $a= 169,173, 179,181$. Let $U_i\cup S_i
(i=1,2)$ be as in Case 14. Let
$$U_3\cup S_3=\{ 161,155,159,164  \} \cup \{ 163,167,169,173,179\} .$$

\noindent {\bf Case 17:} $a=187$. Let
$$U_1\cup S_1= \{161,155,159,158   \} \cup \{143,149,163,167,173,179 \} ,$$
$$U_2\cup S_2= \{133,145,117,142  \} \cup \{127,137,157,187\} ,$$
$$U_3\cup S_3= \{119,115,123,124  \} \cup \{121,131,139,151  \} ,$$
$$U_4\cup S_4= \{185,183,178\} \cup \{169,181\}.$$

\noindent {\bf Case 18:} $a=191,193,197$. Let
$$U_1\cup S_1= \{161,185,177,176   \} \cup \{149,167,173,179,181,191,197 \} ,$$
$$U_2\cup S_2= \{133,145,153,148   \} \cup \{121,151,157,163,169,193  \} ,$$
$$U_3\cup S_3= \{119,125,129,134   \} \cup \{127,131,137,139,143  \} .$$
If $a=197,191$, let$$U_4\cup S_4= \{155,183,182 \} \cup \{187
\}.$$ If $a=193$, let$$U_4\cup S_4= \{175,183,178\} \cup \{187  \}
.$$

\noindent {\bf Case 19:} $a=199$. Let
$$U_1\cup S_1= \{161,145,141,146   \} \cup \{131,137,149,173 \} ,$$
$$U_2\cup S_2= \{133,125,123,128  \} \cup \{121,127\} ,$$
$$U_3\cup S_3= \{119,115,129,164  \} \cup \{139,143,179,199  \} ,$$
$$U_4\cup S_4= \{185,189,194 \} \cup \{187,191,193,197\},$$
$$U_5\cup S_5= \{175,171,172 \} \cup \{151,157,163,167,169,181\}.$$

This completes the proof of Conjecture \ref{conj2} for $k=4$.

\section{Proof of Theorem \ref{mainthm2} and Conjecture \ref{conj1} for
$k=4$}\label{k=4}

Let
$$A(48, 4)=(F_4\setminus \{ 7 \})\cup \{ 11\} .$$
Then $A(48,4)$ does not contain $5$ pairwise coprime integers,
$|A(48, 4)|=|F_4\cap [1, 48]|$ and $A(48, 4)\not=F_4\cap [1, 48]$.

 We will prove that, if $|A(n,4)|\ge |F_4\cap [1, n]|$, then $|A(n,
4)|=|F_4\cap [1, n]|$ for $7\le n\le 48$ and  $A(n, 4)=F_4\cap [1,
n]$ for $49\le n\le 199=210-11$.

Let $W_1$ be the set of all primes together with 1.
$$ W_2=\{ 11^2,  7^2, 5^2, 3^2, 2^2\},\quad  W_3
=\{ 11\times 13, 7\times 19, 5\times 23, 3\times 17, 2^5\},$$
$$ W_4=\{ 11\times 17, 7\times 13, 5^3, 3^3, 2^3\},\quad  W_5
=\{ 13^2, 7\times 11, 5\times 17, 3^4, 2^4\},$$
\begin{eqnarray*}
W_6=[1,n]\setminus (W_1\cup \cdots \cup W_5) =[1,n]\cap (F_4
\setminus Y_1),\end{eqnarray*} where
$$Y_1=\{ 2, 3, 5, 7, 49, 25, 9, 4, 133, 115,
51, 32, 91, 125, 27, 8,77, 85, 81, 16\} .$$ It is clear that
 for each $1\le i\le 5$, the integers in $W_i$
are pairwise coprime. Hence, for $7\le n\le 199$ we have
$$|A(n,4)\bigcap W_1|\le 4=|[1,n]\cap \{ 2, 3, 5, 7\} | \text{ for } n\ge 7,$$
$$|A(n,4)\bigcap W_2|\le |[1,n]\cap \{ 49, 25, 9, 4\} |,$$
$$|A(n,4)\bigcap W_3|\le |[1,n]\cap \{ 133, 115, 51, 32\} |,$$
$$|A(n,4)\bigcap W_4|\le |[1,n]\cap \{ 91, 125, 27, 8\} |,$$
$$|A(n,4)\bigcap W_5|\le |[1,n]\cap \{ 77, 85, 81, 16\} |,$$
\begin{eqnarray*}|A(n,4)\cap W_6|\le  |W_6|
=|F_4\cap [1,n]|-|[1,n]\cap Y_1 |,\end{eqnarray*} Since $[1,
n]\subseteq \bigcup_{i=1}^6 W_i$, we have
\begin{equation}\label{mainthm2lem}|A(n,4)|=\sum_{i=1}^6 |A(n,4)\cap W_i|\le |F_4\cap
[1,n]|.\end{equation} Hence, if $7\le n\le 199$ and $|A(n,4)|\ge
|F_4\cap [1, n]|$, then $|A(n,4)|=|F_4\cap [1, n]|$. This implies
that $f(n,4)=|E(n,4)$ for all $n\ge 7$. So Conjecture \ref{conj1}
for $k=4$ is true.

Now we assume that $49\le n\le 199$. By \eqref{mainthm2lem} we have
$$|A(n,4)\cap W_1|=4,\quad |A(n,4)\cap W_2|=4.$$
By $|A(n,4)\cap W_2|=4$ we know that if $(m, 2\times 3\times 5\times
7\times 11)=1$ and $1\le m\le 199$, then $m\notin A(n,4)$. Thus
\begin{equation*}A(n,4)\cap W_1\subseteq \{ 2, 3, 5, 7, 11\}, \end{equation*}
\begin{equation}\label{w}
A(n,4)\subseteq (F_4\cup \{ 11, 121, 143, 187\})\cap
[1,n].\end{equation}

Suppose that $A(n,4)\cap \{ 11,  121, 143,187\} \not=\emptyset $.

Since $A(n,4)\cap W_1\subseteq \{ 2, 3, 5, 7, 11\}$ and
$|A(n,4)\cap W_1|=4$, there exists a prime $p\in \{ 2, 3, 5, 7\}$
such that $A(n,4)\cap W_1=\{ 2, 3, 5, 7, 11\}\setminus \{ p\}$,
then $ p, p^2, 13p, 17p, 19p\notin A(n,4)$. Since $n\ge 49$,
$p<p^2\le n$, $13p<17p<121$ and $19p<143$, we have
$$| \{ 11,  121, 143,187\}\cap [1, n]|< |\{ p, p^2,
13p, 17p, 19p\} \cap [1, n]|.$$ Hence \begin{eqnarray*}|A(n,4)|&\le&
|(F_4\cup \{ 11, 121, 143, 187\})\cap [1,n]|\\
&& -|\{ p, p^2,
13p, 17p, 19p\}\cap [1,n]|\\
&=& |F_4\cap [1,n]|+|\{ 11, 121, 143, 187\})\cap [1,n]|\\
&& -|\{ p, p^2,
13p, 17p, 19p\}\cap [1,n]|\\
&<& |F_4\cap [1,n]|,\end{eqnarray*} a contradiction.

Hence $A(n,4)\cap \{ 11,  121, 143,187\} =\emptyset $. By
$|A(n,4)|\ge |F_4\cap [1,n]|$ and \eqref{w} we have
$A(n,4)=F_4\cap [1,n]$ for $49\le n\le 199$.

Now Theorem \ref{mainthm2} follows from  Theorem \ref{generalthm}
and Conjecture \ref{conj2}.

\section{Final Remarks}

It seems that the method in this paper can be used for large $k$
with more complicated arguments. We pose several problems for
further research.

\begin{problem} What is the smallest positive integer $k$ for which Erd\H os
original conjecture is false? \end{problem}

\begin{problem} Is Erd\H os
original conjecture true or false for infinitely many positive
integers $k$?
\end{problem}

\begin{problem} Is
$$\limsup_{k\to\infty} \sup_{n\ge 1} (f(n,k)-E(n,k))<+\infty ?$$
In particular, is
$$\limsup_{k\to\infty} \sup_{n\ge 1} (f(n,k)-E(n,k))=1?$$

\end{problem}

\begin{problem} Is $E(k)=p_k^2$ true or false for infinitely many $k$?
\end{problem}

Let $p_i$ be the $i-$th prime. Ahlswede and  Khachatrian
\cite{Ahlswede} proved that if $$(H) \quad p_{t+7}p_{t+8}\le
n<p_tp_{t+9}, \quad p_{t+9}<p_t^2,$$ then for $k=t+3$,
$$f(n,k)>|E(n,k)|.$$

As remark in \cite{Ahlswede}, (H) holds for $t=209$. We can verfy
that second such $t$ is 1823. We pose the following conjecture.

\begin{conjecture} The set of $k$ for which Erd\H os
original conjecture is false has the density zero.\end{conjecture}

\section{Acknowledgements}We would like to thank the referee for his/her
comments.


\begin{thebibliography}{30}

\bibitem{Ahlswede} R. Ahlswede and L. H. Khachatrian, {\it On extremal sets without coprimes,}  Acta Arith. 66(1994), 89-99.

\bibitem{Erdos1962} P. Erd\H os, Remarks in number theory, IV, Mat. Lapok 13(1962), 228-255.

 \bibitem{Erdos1965} P.  Erd\H os,  Extremal  problems  in  number  theory,  Proc.  Sympos.  Pure  Math.,  vol.  8,
Amer. Math.  Soc,  Providence,  R.  I.,  1965, pp.  181-189.

\bibitem{Choi} S.  L.  G.  Choi, On sequences  containing at most 3 pairwise coprime integers, Trans. Amer. Math. Soc. 183(1973), 437-440.

\bibitem{Toth} C. Szab\'o and G. T\'oth, Maximal sequences not
containing 4 pairwise coprime integers, Mat. Lapok 32(1985),
253-257(in Hungarian).
\end{thebibliography}
\end{document}